\documentclass[english]{amsart}
\usepackage[latin1]{inputenc}
\usepackage{amsthm,amsmath,amssymb,mathrsfs}
\usepackage{hyperref}
\theoremstyle{plain}
\newtheorem{thm}{Theorem}[section]

\newtheorem{prop}[thm]{Proposition}

\newtheorem{lem}[thm]{Lemma}

\newtheorem{cor}[thm]{Corollary}
\theoremstyle{definition}
\newtheorem{prob}[thm]{Problem}

\theoremstyle{remark}

\newcommand{\eqdef}{\stackrel{\rm def}{=}}
\usepackage{babel}

\DeclareMathOperator{\inv}{inv}
\DeclareMathOperator{\Des}{Des}
\DeclareMathOperator{\HDes}{HDes}
\DeclareMathOperator{\maj}{maj}
\DeclareMathOperator{\fmaj}{fmaj}
\DeclareMathOperator{\Aut}{Aut}
\newcommand{\si}{\sigma}

\newcommand{\eps}{\epsilon}

\author{Fabrizio Caselli}
\title{Signed mahonians on some trees and parabolic quotients}

\begin{document} 
\maketitle
\begin{center}
   \emph{\footnotesize Dipartimento di matematica, Universit\`a di Bologna,\\ P.zza di Porta S. Donato 5, Bologna 40126, Italy}
\end{center}

\begin{abstract}
We study the distribution of the major index with sign on some parabolic quotients of the symmetric group, extending and generalizing simultaneously results of Gessel--Simion and Adin--Gessel--Roichman, and on the labellings of some special trees that we call rakes. We further consider and compute the distribution of the flag-major index on some parabolic quotients of wreath products and other related groups. All these distributions turn out to have very simple factorization formulas.
\end{abstract}

\section{Introduction} Let $W$ be a finite reflection group of rank $n$ and $S$ be a set of simple reflections for $W$ as a Coxeter group. If $\ell$ denotes the length function on $W$ with respect to $S$, the distribution of $\ell$ on $W$ is called the Poincar\'e polynomial and it is a classical result of  Chevalley \cite{Ch} and Solomon \cite{So} in different level of generalizations that
\begin{equation}\label{mahonian}
\sum_{u\in W}q^{\ell(u)}=[d_1]_q[d_2]_q\cdots[d_n]_q,
\end{equation}
where $d_1,\ldots,d_n$ are the fundamental degrees of $W$. The main problems faced in this work are variations of this identity (with statistics other than the length function) to parabolic quotients of $W$ and labellings of trees in the sense of \cite{BW}. 

If $J\subset S$ we denote by $W_J$ and $\textrm{\raisebox{6pt}{$_J$}}\!W$ the corresponding parabolic subgroup and (left) parabolic quotient. It is well-known that given $u\in W$ there exists a unique decomposition $
u=u_J\cdot \textrm{\raisebox{5pt}{$_J$}}\!u,$ where $u_J\in W_J$ and $\textrm{\raisebox{5pt}{$_J$}}\!u\in\, \textrm{\raisebox{5pt}{$_J$}}\!W$; it is also well-known that in this decomposition we have 
\begin{equation}\label{additi}\ell(u)=\ell(u_J)+\ell(\textrm{\raisebox{5pt}{$_J$}}\!u).
\end{equation}
It follows from Equation \eqref{additi} that  the distribution of the length function on the parabolic quotient $\textrm{\raisebox{6pt}{$_J$}}W$ is given by
\begin{equation}\label{decom}
\sum_{u\in \textrm{\raisebox{5pt}{$_J$}}\!W}q^{\ell(u)}=\frac{\sum_{u\in W}q^{\ell(u)}}{\sum_{u\in W_J}q^{\ell(u)}}. 
\end{equation}
In particular, if $W=S_n$ is the symmetric group on $n$ letters, $S$ is identified with the set of positive integers $\{1,2,\ldots,n-1\}$ and $J=\{n-k+1,\ldots,n-1\}$, then $W_J$ is clearly isomorphic as a Coxeter group to $S_k$ and therefore
$$
\sum_{u\in \mathscr S_{n,k}}q^{\ell(u)}=[k+1]_q[k+2]_q\cdots[n]_q,
$$
where $\mathscr S_{n,k}\eqdef \raisebox{5pt}{$_J$}\!W$, since the fundamental degrees of $S_n$ are $2,3,\ldots,n$.

It is  a classical result of MacMahon \cite{M} that the major index is equidistributed with the length function on $S_n$, i.e. $\sum_{\si\in S_n}q^{\maj(\si)}=\sum_{\si\in S_n}q^{\ell(\si)}$.  As one can easily verify, Equation \eqref{additi} is no longer satisfied with $\maj$ in the place of $\ell$. Nevertheless, as an immediate consequence of classical results of Stanley \cite{St} and Foata-Sch\"utzenberger \cite{FS} the statistics $\maj$ and $\inv$ remain equidistributed on all parabolic quotients $\textrm{\raisebox{6pt}{$_J$}}\!W$, and in particular
\begin{equation}\label{greta}
\sum_{u\in \mathscr S_{n,k}}q^{\maj(u)}=[k+1]_q[k+2]_q\cdots[n]_q.
\end{equation}
Major index and inversion number can also be defined for labellings $w\in\mathscr W(F)$ of a forest $F$ (see \cite{BW}). In particular, one can consider a particular tree $R_{n,k}$, that we call a rake, and the set of equivalence classes of labellings $\mathscr R_{n,k}$, by the action of the automorphism group of $R_{n,k}$ (see section \ref{rakes} for more details), such that the distribution of the major index is still
\begin{equation}\label{bjwa}
 \sum_{w\in \mathscr R_{n,k}}q^{\maj(w)}=[k+1]_q[k+2]_q\cdots[n]_q.
\end{equation}
We observe that Equations \eqref{greta} and \eqref{bjwa} are trivially equivalent for $k=1$ only, while there is no bijective explanation for this equidistribution for $k>1$.
The main target of this work is to study the signed versions of the distribution of $\maj$ on $\mathscr S_{n,k}$ and on $\mathscr R_{n,k}$, i.e. we study the polynomials
\[
s_{n,k}(q)\eqdef \sum_{\si\in \mathscr S_{n,k}}(-1)^{\inv(\si)} q^{\maj(\si)}
\qquad \textrm{and}\qquad
  r_{n,k}(q)\eqdef\sum_{w\in \mathscr R_{n,k}}(-1)^{\inv(w)} q^{\maj(w)}.
\]

If $k=1$  these polynomials are computed by a well-known formula of Gessel and Simion (see \cite{Wa} for an elegant bijective proof based on its unsigned version \eqref{mahonian}):
$$
s_{n,1}(q)=r_{n,1}(q)=\sum_{\si\in S_n}(-1)^{\inv(\si)} q^{\maj(\si)}=[2]_{-q}[3]_{q}\cdots[n]_{(-1)^{n-1}q},
$$
which is a sort of an alternating version of Equation \eqref{mahonian}, and the first main result here is the following alternating version of Equation \eqref{greta}  that include the Gessel-Simion formula as a special case
\begin{equation}\label{gessimgen}
  s_{n,k}(q)=[k+1]_{(-1)^{nk+n+k}q}[k+2]_{(-1)^{k+1}q}[k+3]_{(-1)^{k+2}q}\cdots[n]_{(-1)^{n-1}q}.
\end{equation}
As for the polynomial $r_{n,k}(q)$, we show that $r_{n,k}(q)=s_{n,k}(q)$ unless $n$ is odd and $k$ is even, (in which case $s_{n,k}(q)$ does not seem to factorize nicely at all) strengthening the fact that these two families of combinatorial objects are strictly related.
The nice combinatorial and bijective methods used in \cite{FS,St,Wa} for the corresponding unrestricted results can not be easily generalized to the present context in the computation of $s_{n,k}(q)$ whilst an idea appearing in \cite{AGR} will be of some help. We also mention here that the present proof of Equation \eqref{gessimgen} is rather involved and full of technicalities (many of which will be omitted), so that one could say that it proves the result without really explaining it; and it would be really desirable to find an algebraic explanation for it, or at least a simple combinatorial proof.

Another possible generalization of Equation \eqref{greta} considers special classes of ``parabolic'' subgroups for complex reflection groups. In fact, if $W$ is a complex reflection group, although one can define a ``length function" with respect to some generating set of pseudo-reflections, this concept lacks algebraic significance and in particular Equation \eqref{mahonian} is no longer valid. Nevertheless, for wreath products $G(r,n)$  of the cyclic group $C_r$ with $S_n$, and in particular for Weyl groups of type $B$, there is a natural counterpart of the major index called the flag-major index. This index has been introduced in  \cite{AR} and has the following distribution
$$
\sum_{g\in G(r,n)}q^{\fmaj(g)}=[d_1]_q[d_2]_q\cdots[d_n]_q,
$$
where $d_i=ri, i=1,2,\ldots,n$ are the fundamental degrees of $G(r,n)$; it is therefore natural to ask whether one can extend Equation \eqref{greta} to wreath products $G(r,n)$. This problem is solved in this paper as a particular case of a more general result involving a wider class of groups and using some machinery developed by the author in \cite{Ca1} in the study of some aspects of the invariant theory of complex reflection groups. 

\section{Notation and preliminaries}\label{nota}
In this section we collect the notations that are used in this paper.
If $r,n\in \mathbb N$ we let $[n]\eqdef\{1,2,\ldots,n\}$ and $\mathbb Z_r\eqdef \mathbb Z /r\mathbb Z$. We let $\mathcal P _n\eqdef \{\lambda=(\lambda_1,\ldots,\lambda_n)\in \mathbb N^n:\, \lambda_1\geq \lambda_2 \geq \cdots \geq \lambda_n\}$ be the set of \emph{partitions} of \emph{length} at most $n$.
If $q$ is an indeterminate and $n\in \mathbb N$ we let $[n]_q\eqdef\frac{1-q^n}{1-q}=1+q+q^2+\cdots+q^{n-1}$ be the $q$-\emph{analogue} of $n$ and $[n]_q!\eqdef [1]_q[2]_q\cdots [n]_q$. 
A permutation $\si \in S_n$ will be denoted by $\si=[\si(1),\cdots, \si(n)].$ We denote by $\inv(\si)\eqdef|\{(i,j):i<j\textrm{ and } \si(i)>\si(j)\}$ the number of \emph{inversions} of $\si$ and by $\maj(\si)\eqdef \sum_{i:\si(i)>\si(i+1)}i$  the \emph{major index} of $\si$.

According to \cite{BW} we say that a poset $P$ is a \emph{forest} if every element of $P$ is covered by at most one element. If $P$ is a finite forest with $n$ elements we let $\mathscr W(P)\eqdef \{w:P\rightarrow \{1,2,\ldots,n\}\textrm{ such that $w$ is a bijection}\}$ be the set of \emph{labellings} of $P$. If $w$ is a labelling of a forest $P$ we say that a pair $(x,y)$ of elements of $P$ is an inversion of $w$ if $x<y$  and $w(x)>w(y)$ and we let $\inv(w)$ be the number of inversions of $w$; we say that $x\in P$ is a \emph{descent} of $w$ if $x$ is covered by an element $y$ and $w(x)>w(y)$, and we denote by $\Des(w)$ the set of descents of $w$; if $x\in P$ we let $h_x=|\{a\in P:a\leq x\}|$ be the \emph{hook length} of $x$ and 
\[\maj(w)=\sum_{x\in \Des(w)}h_x
\]
be the \emph{major index} of $w$. Finally, if $w$ is a labelling of a forest $P$ we let 
\[\mathscr L(w)=\{\sigma\in S_n:\, \textrm{if $x<y$ then $\sigma^{-1}(w(x))<\sigma^{-1}(w(y))$} \}
\] 
be the set of \emph{linear extensions} of $w$. The main results in \cite{BW} we are interested in are summarized in the following result.
\begin{thm}
   If $P$ is a forest then
\begin{equation}\label{bw1}
   \sum_{w\in \mathscr W(P)}q^{\maj(w)}=\frac{n!}{\prod_{x\in P}h_x}\prod_{x\in P}[h_x]_q
\end{equation}
and if $w\in \mathscr W(P)$ then 
\begin{equation}\label{bw2}
   \sum_{\sigma \in \mathscr L(w)}q^{\maj(\sigma)}=q^{\maj(w)}\frac{[n]_q!}{\prod_{x\in P}[h_x]_q}.
\end{equation}
\end{thm}

If $r>0$,  a \emph{$r$-colored integer} is a pair $(i,z)$, denoted also $i^z$, where $i\in \mathbb N$ and $z\in \mathbb Z_r$, and we let $|i^z|\eqdef i$. The {\em wreath product} $G(r,n)\eqdef\mathbb Z_r \wr S_n$ is the group of permutations $g$ of the set of $r$-colored integers $i^z$, where $i\in [n]$ and $z\in \mathbb Z_r$ such that, if $g(i^0)=j^z$ then $g(i^{z'})=j^{z+z'}$. An element $g\in G(r,n)$ is therefore uniquely determined by the $r$-colored integers $g(1^0), \ldots, g(n^0)$ and we usually write $g=[\si_1^{z_1},\ldots,\si_n^{z_n}]$, where $g(i^0)=\si_i^{z_i}$. Note that in this case we have that $|g|\eqdef[\si_1, \ldots, \si_n]\in S_n$. When it is not clear from the context, we will denote $z_i$ by $z_i(g)$. 

For $p|r$ the {\em complex reflection group} $G(r,p,n)$ is the normal subgroup of $G(r,n)$ defined by
\begin{equation}\label{def-grpn}
G(r,p,n):=\{[\si_1^{z_1},\ldots,\si_n^{z_n}]\in G(r,n) \mid z_1+\cdots+z_n\equiv 0 \mod p\}.
\end{equation}
and its \emph{dual group} 
\begin{equation}\label{def-grpqn}
G(r,p,n)^*=G(r,n)/ C_p,
\end{equation}
where $C_p$ is the cyclic subgroup of $G(r,n)$ of order $p$ generated by $[1^{r/p},2^{r/p},\ldots,n^{r/p}]$.

The study of permutation statistics  has found a new interest in the more general setting of complex reflection groups after the work of Adin and Roichman \cite{AR}. Some of these results have been generalized in \cite[\S 5]{Ca1} in the following way.

For $g=[\si_1^{z_1},\ldots,\si_n^{z_n}]\in G(r,p,n)^*$ we let
\begin{align*}
\HDes(g)&\eqdef \{i\in [n-1]:\, z_i=z_{i+1} \textrm{, and }\si(i)>\si(i+1)\}\\
h_i(g)&\eqdef \#\{j\geq i:\,j\in \HDes(g)\}.
\end{align*}
We also let $(k_1(g),\ldots,k_n(g))$ be the smallest  element in $\mathcal P_n$ (with respect to the entrywise order) such that $g=[\si_1^{k_1(g)},\ldots,\si_n^{k_n(g)}]$, where we make slight abuse of notation identifying an integer with its residue class in $\mathbb Z_r$. In other words, $(k_1(g),\ldots,k_n(g))$ is a partition characterized by the following property: if $(\beta_1, \ldots, \beta_n)\in \mathcal P_n$  is such that $g=[\si_1^{\beta_1},\ldots,\si_n^{\beta_n}]$, then  $\beta_i\geq k_i(g)$, for all $i\in[n]$. 
For example, let $g=[2^2,7^3,6^3,4^5,8^1,1^1,5^3,3^2]\in G(6,3,8)^*$. Then $\HDes(g)=\{2,5\}$, $(h_1,\ldots,h_8)=(2,2,1,1,1,0,0,0)$ and $(k_1,\ldots,k_8)=(18,13,13,9,5,5,1,0)$.

We note that if we let $\lambda_i(g)\eqdef r\cdot h_i(g)+k_i(g),$ then the sequence $\lambda(g)\eqdef(\lambda_1(g),\ldots,\lambda_n(g))$ is a partition such that $g=[\si_1^{\lambda_1(g)},\ldots,\si_n^{\lambda_n(g)}]$. The {\em flag-major index} of an element $g\in G(r,p,n)^*$ is defined by  $\fmaj(g)\eqdef |\lambda(g)|.$
All these definitions are valid for wreath products $G(r,n)$ just by letting $p=1$ and one can easily verify that for $r=p=1$ $\fmaj(\si)=\maj(\si)$ for all $\si\in S_n$.

\section{Signed mahonians in parabolic quotients of symmetric groups}

For $n>0$ and $k=0,\ldots, n$ we let $J=\{n-k+1,\ldots,n-1\}$ and, following \cite[\S 2.4]{BB}, we let $$\mathscr S_{n,k}\eqdef \textrm{\raisebox{6pt}{$_J$}}\!S_n=\{\sigma\in S_n:\, \sigma^{-1}(n-k+1)<\sigma^{-1}(n-k+2)<\cdots<\sigma^{-1}(n)\},$$
be the corresponding left parabolic quotient.
We observe that if $k=0,1$ then $J=\emptyset$ and so $\mathscr S_{n,k}=S_n$; moreover, if $\si\in \mathscr S_{n,k}$ then $\si(n)\in\{1,2,\ldots,n-k,n\}$ and we set
$$
s(\si)\eqdef\begin{cases} \si(n)-1,&\textrm{if }\si(n)\in[n-k];\\ n-k,&\textrm{otherwise.}\end{cases}
$$
We consider the following generating function
$$
s_{n,k}(q,z)\eqdef\sum_{\si\in \mathscr S_{n,k}}(-1)^{\inv(\sigma)}q^{\maj(\si)}z^{s(\si)}.
$$
We will be mainly interested in the special evaluation $s_{n,k}(q)\eqdef s_{n,k}(q,1)$ but we will eventually find an explicit formula for the whole bivariate generating function $s_{n,k}(q,z)$.
Observe that, by definition, $s_{n,0}(q,z)=s_{n,1}(q,z)$ and we recall that this generating function has already been computed by Adin-Gessel-Roichman \cite{AGR}. We also extend an idea appearing in \cite{AGR} to prove the following result, where we let $\eps\eqdef -1$.

\begin{thm}\label{recur}Let $k\in [n-1]$. Then
$$s_{n,k}(q,z)=\frac{1}{1+z}\Big(\big(\eps^k z^{n-k} +(-q)^{n-1}\big)s_{n-1,k}(q,1)+\eps^n z (1-q^{n-1}) s_{n-1,k}(q,-z)\Big)+z^{n-k} s_{n-1,k-1}(q,1).$$
\end{thm}
\begin{proof}
We need to construct the set $\mathscr S_{n,k}$ from analogous sets in $S_{n-1}$. For this it will be helpful the following notation: if $\si\in S_n$ we let $\si_0\in S_{n-1}$ given by
$$
\si_0(i)\eqdef\begin{cases}\si(i),&\textrm{if }\si(i)<\si(n);\\ \si(i)-1,& \textrm{if }\si(i)>\si(n),\end{cases}
$$
for all $i\in [n-1]$. In other words $\si_0$ is obtained from $\si$ by deleting the last entry and rescaling  the others.
Now it is clear that the map $\si\mapsto \si_0$ establishes a bijection between $\{\si\in \mathscr S_{n,k}:\, \si(n)=n\}$ and $\mathscr S_{n-1,k-1}$, and also between the sets $\{\si\in \mathscr S_{n,k}:\,\si(n)=i<n\}$ and $\mathscr S_{n-1,k}$, for all $i\in [n-k]$. So, the map $\sigma\mapsto(\sigma_0,\si(n))$ establishes an explicit bijection
$$
\mathscr S_{n,k}\longleftrightarrow  (\mathscr S_{n-1,k}\times [n-k])\sqcup (\mathscr S_{n-1,k-1}\times\{n\}) ,
$$
where $\sqcup$ denotes disjoint union. In this bijection, if $\sigma\leftrightarrow (\sigma_0,i)$, with $\sigma_0\in \mathscr S_{n-1,k}$ and $i\in[n-k]$, then
\[
\begin{cases}
\inv(\sigma)=\inv(\sigma_0)+n-i\\ \maj(\si)=\begin{cases}\maj(\si_0)+n-1& \textrm{if }i\leq \si_0(n-1)\\ \maj(\si_0)&\textrm{if }i> \si_0(n-1)\end{cases}\\ s(\si)=i-1,
\end{cases}
\]

and if $\sigma\leftrightarrow (\si_0,n)$ with $\si_0\in \mathscr S_{n-1,k-1}$ then
\[\begin{cases}\inv(\sigma)=\inv(\sigma_0)\\ \maj(\si)=\maj(\si_0)\\ s(\si)=n-k\end{cases} .
\]

We use the bijection above and these equations to compute recursively the polynomial $s_{n,k}(q,z)$: we easily obtain

\[
   s_{n,k}(q,z)=\sum_{\si_0\in \mathscr S_{n-1,k}}\eps^{\inv(\si_0)+n-1}q^{\maj(\si_0)}\Big(q^{n-1}\sum_{j=0}^{s(\si_0)}(\eps z)^j+\sum_{j=s(\si_0)+1}^{n-k-1}( \eps z)^j\Big)+z^{n-k}s_{n-1,k-1}(q,1).
\]
Computing the geometric sums on the right-hand side we then conclude that
\begin{align*}
s_{n,k}(q,z)&=\sum_{\si\in \mathscr S_{n-1,k}}\eps^{\inv(\si)+n-1}q^{\maj(\si)}\Big(q^{n-1}\frac{1-(\eps z)^{s(\si)+1}}{1-\eps z}+\frac{(\eps z)^{s(\si)+1}-(\eps z)^{n-k}}{1-\eps z}\Big)+z^{n-k}s_{n-1,k-1}(q,1)\\
&=\frac{1}{1+z}\Big(\big(\eps^k z^{n-k} +(-q)^{n-1}\big)s_{n-1,k}(q,1)+\eps^n z (1-q^{n-1}) s_{n-1,k}(q,\eps z)\Big)+z^{n-k} s_{n-1,k-1}(q,1).\qedhere
\end{align*}
\end{proof}
Theorem \ref{recur} can be used to compute the polynomials $s_{n,k}(q,z)$ taking as initial condition $s_{n,n}(q,z)=1$ for all $n>0$ and recalling that $s_{n,0}(q,z)=s_{n,1}(q,z)$. As mentioned in the introduction, the special evaluation of the polynomials $s_{n,k}(q,z)$ at $z=1$ will have the following nice factorization
$$
s_{n,k}(q)=[k+1]_{\eps^{k+n+nk}q} [k+2]_{\eps^{k+1}q} [k+3]_{\eps^{k+2}q}\cdots [n]_{\eps^{n-1}q}.
$$
In particular, for $k=1$, we find the Gessel-Simion formula $s_{n,1}(q,1)=[2]_{-q}[3]_q\cdots[n]_{\eps^{n-1}q}$.

Unfortunately, the polynomials $s_{n,k}(q,z)$ do not factorize in general as nicely as their specializations at $z=1$ (at least if $k$ is even).
 Nevertheless, Theorem \ref{recur} allows us to prove explicit formulas for these polynomials. We start with a simple example.
\begin{lem}\label{n,n-1}
For all $n>0$ we have $s_{n,n-1}(q,z)=z[n-1]_{-q}+(-q)^{n-1}$ and in particular $s_{n,n-1}(q,1)=[n]_{-q}$.
\end{lem}
\begin{proof}
This can be proved by induction using Theorem \ref{recur} but it is simpler to provide a direct proof. In fact we clearly have $\mathscr S_{n,n-1}=\{[1\,2\cdots n], [2\,1\,3\cdots n],\ldots,[2\,3\cdots n-1 \,1 \,n],[2\,3\cdots n\, 1]\}$ and the result follows immediately from the definition.
\end{proof}
In the following result we have a completely explicit description of the polynomials $s_{n,k}(q,z)$.
\begin{thm}\label{main}If $k<n$ is odd we have
$$s_{n,k}(q,z)=[k+1]_{-q} [k+2]_{q}\cdots [n-1]_{\eps^{n}q}\Big(\sum_{i=0}^{n-k-1}\eps^{(n+1)(n-i-1)}z^i q^{n-i-1}+z^{n-k}[k]_{\epsilon^{n-1}q}\Big).
$$
If $k<n-1$ is even we have
\begin{align*}s_{n,k}(q,z)=&\,\,[k+2]_{-q}\cdots [n-1]_{\eps^{n}q}\cdot
\Big([k+1]_{\eps^n q}[n]_{\eps^{n-1}q}+(z-1)\\&\cdot\Big(\sum_{i=0}^{n-k-1}[k+1]_{\eps^n q}[n-i-1]_{\eps^{n+1}q}z^i+\sum_{\substack{i=0\\ i\textup{ even}}}^{n-k-1}q^{n-i-1}z^i\big([k]_{-q}- [k]_{q}\big)\Big)\Big),
\end{align*}
where the last sum runs through all nonnegative even integers smaller than $n-k$.
\end{thm}
\begin{proof}
The result is readily verified for $n=1,2,3$. We may also easily verify that the claimed expression for $k=1$ agrees with the one for $k=0$. These initial conditions together with the recurrence given by Theorem \ref{recur} uniquely determine the polynomials $s_{n,k}(q,z)$, and hence we only have to verify that the claimed expressions actually satisfy this recurrence.

We provide some sketches of the proof only if $k\neq n-2$ is even. If $k$ is odd or $k=n-2$ the proof is similar (and simpler) and is left to the reader.

So let $2\leq k\leq n-3$, with $k$  even. If we  substitute the claimed expressions in the right-hand side of the recursion in Theorem \ref{recur}, and we delete the factor let $[k+2]_{-q}\cdots[n-1]_{\eps^n q}$ for notational convenience, we obtain the polynomial
\begin{align*}
g_{n,k}(q,z)
&\eqdef \frac{1}{1+z}(z^{n-k} -\eps^n q^{n-1}+\eps^n z-\eps^n zq^{n-1})[k+1]_{\eps^{n-1}q}-\eps^n z(1-\eps^n q)\\&\hspace{0cm}\cdot\Big(\sum_{i=0}^{n-k-2}[k+1]_{\eps^{n-1}q}[n-i-2]_{\eps^n q}(-z)^i+\sum_{\substack{i=0\\ i\textrm{ even}}}^{n-k-2}q^{n-i-2}z^i([k]_{-q}- [k]_{q})\Big)+z^{n-k} [k]_{-q}[k+1]_{q},
\end{align*}
where we have also used that $1-q^{n-1}=(1-\eps^n q)[n-1]_{\eps^n q}$. So the proof will be achieved if we show that
\begin{equation}\label{gnk}
   g_{n,k}(q,z)=[k+1]_{\eps^n q}[n]_{\eps^{n-1}q}+(z-1)\cdot\Big(\sum_{i=0}^{n-k-1}[k+1]_{\eps^n q}[n-i-1]_{\eps^{n+1}q}z^i+\sum_{\substack{i=0\\ i\textrm{ even}}}^{n-k-1}q^{n-i-1}z^i\big([k]_{-q}- [k]_{q}\big)\Big).
\end{equation}
By means of the identity $(1-\eps^n q)([k]_{-q}-[k]_{q})=-2q[k]_{\eps^{n-1}q}$ we can deduce that
\begin{align*}
g_{n,k}(q,z)&=(\eps^n z+\eps^{n-1}z^2+\cdots+z^{n-k-1}-\eps^n q^{n-1})[k+1]_{\eps^{n-1}q}\\&-\eps^n z\Big(\sum_{i=0}^{n-k-2}[k+1]_{\eps^{n-1}q}(1-(\eps^n q)^{n-i-2})(-z)^i-2\sum_{\substack{i=0\\ i \textrm{ even}}}^{n-k-2}q^{n-i-1}z^i[k]_{\eps^{n-1}q})\Big)+z^{n-k} [k]_{-q}[k+1]_{q}.
\end{align*}

Now we make the following observation. If  we expand $g_{n,k}(q,z)=\sum_{i\geq 0} a_i(q)z^i$  then we also have $g_{n,k}(q,z)=b_{-1}(q)+(z-1)\sum_{i\geq 0} b_i (q)z^i$ where
 \begin{equation}\label{aibi}b_i(q)=\sum_{j>i}a_j(q) \,\,\textrm{ for all }i\geq -1.
 \end{equation}
To complete the proof we only have to compute the polynomials $b_i(q)$ using Equation \eqref{aibi} and verify that they agree with the expressions given in \eqref{gnk}. Unfortunately,  we need to split the proof again, according to the parity of $i$ and $n$. We treat the case where $i$ and $n$ are both even, leaving the other similar cases to the reader.

We have
\begin{align*}
b_i(q)&=[k+1]_{-q}\Big(1-\sum_{j=i}^{n-k-2}(1-q^{n-j-2})(-1)^j\Big)+2[k]_{-q}\sum_{\substack{j=i\\ j \textrm{ even}}}^{n-k-2}q^{n-j-1}+[k]_{-q}[k+1]_q\\
&=[k+1]_{-q}(q^k-q^{k+1}+\cdots+q^{n-i-2})+2[k]_{-q}(q^{k+1}+q^{k+3}+\cdots+q^{n-i-1})+[k]_{-q}[k+1]_q\\
&=[k]_{-q}[n-i]_q+q^k[n-i-1]_{-q}+q^{n-i-1}[k]_{-q}.
\end{align*}
Now we make the simple observation that if $r$ and $s$ are both even then
$
[r]_q[s]_{-q}=[r]_{-q}[s]_q
$
and so
\begin{align*}
b_i(q)&=[k]_{q}[n-i]_{-q}+q^k[n-i-1]_{-q}+q^{n-i-1}[k]_{-q}\\
&=[k]_{q}[n-i-1]_{-q}-q^{n-i-1}[k]_q+q^k[n-i-1]_{-q}+q^{n-i-1}[k]_{-q}\\
&=[k+1]_q[n-i-1]_{-q}+q^{n-i-1}([k]_{-q}-[k]_q),
\end{align*}
and the proof is complete.
\end{proof}

\begin{cor}\label{cormain}
We have
$$
s_{n,k}(q)=\sum_{\si\in \mathscr S_{n,k}}\eps^{\inv(\si)}q^{\maj(\si)}=[k+1]_{\eps^{k+n+nk}q} [k+2]_{\eps^{k+1}q} [k+3]_{\eps^{k+2}q}\cdots [n]_{\eps^{n-1}q},
$$
where $\eps=- 1.$
\end{cor}
\begin{proof}
  This follows immediately from Theorem \ref{main}.
\end{proof}
We close this section by observing that Corollary \ref{cormain} can also be interpreted as an alternating version of (a special case of) Equation \eqref{bw2}. In fact, consider the poset $T_{n,k}=\{x_1,x_2,\ldots,x_n\}$ with the ordering given by $x_i<x_j$ if and only if $n-k<i<j$. The Hasse diagram of $T_{n.k}$ is a forest consisting of $n-k$ disjoint vertices and of a linear tree of length $k$ (see Figure \ref{TeR} (left)). If we consider the natural labelling for $T_{n,k}$ given by $w(x_i)=i$, then we clearly have $\mathscr L(w)=\mathscr S_{n,k}$ and so Corollary \ref{cormain} can also be reformulated as
\[
   \sum_{\sigma\in \mathscr L(w)}\eps^{\inv(\sigma)}q^{\maj(\si)}=[k+1]_{\eps^{k+n+nk}q} [k+2]_{\eps^{k+1}q} [k+3]_{\eps^{k+2}q}\cdots [n]_{\eps^{n-1}q}.
\]
\section{Signed mahonian distributions on rakes' labellings}\label{rakes}

For $k<n$ we denote by $R_{n,k}$ the poset consisting of $n$ elements $x_1,\ldots,x_n$ with the ordering given by $x_i<x_j$ if and only if $i,k<j$. The Hasse diagram of $R_{n,k}$ is shown in Figure \ref{TeR} (right) and because of its shape we call $R_{n,k}$ a \emph{rake} with $k$ \emph{teeth}. It is clear that the action of $\Aut(R_{n,k})\cong S_k$ on ${\mathscr W}(R_{n,k})$ by permutations of the labels of the teeth preserves inversion index and major index; therefore we can consider these indices also on the set $ \mathscr R_{n,k}$ of equivalence classes of labellings. 
\setlength{\unitlength}{2pt}
\begin{center} 
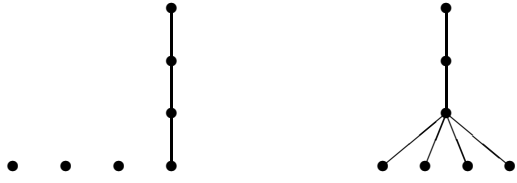
\begin{figure}
\begin{picture}(104,30) 
\put(40,30){\circle*{2}} 
\put(10,0){\circle*{2}}
\put(20,0){\circle*{2}}
\put(30,0){\circle*{2}}
\put(40,0){\circle*{2}}
\put(80,0){\circle*{2}}
\put(88,0){\circle*{2}}
\put(96,0){\circle*{2}}
\put(104,0){\circle*{2}}
\put(40,10){\circle*{2}}
\put(40,20){\circle*{2}}
\put(92,10){\circle*{2}}
\put(92,20){\circle*{2}}
\put(92,30){\circle*{2}}

\put(40,0){\line(0,1){30}}
\put(92,10){\line(0,1){20}}
\put(80,0){\line(6,5){12}}
\put(88,0){\line(2,5){4}}
\put(92,10){\line(2,-5){4}}
\put(92,10){\line(6,-5){12}}

\end {picture}
\caption{$T_{7,4}$ (left) and $R_{7,4}$ (right)}\label{TeR}
\end {figure}
\end{center} 

It follows in particular from \eqref{bw2} that 
\begin{equation}\label{rakeunsigned}
\sum_{w\in {\mathscr R}_{n,k}}q^{\maj(w)}=[k+1]_q\cdots[n]_q,
\end{equation}
and we denote by $r_{n,k}(q)$ its signed version
\[
   r_{n,k}(q)=\sum_{w\in {\mathscr R}_{n,k}}\eps^{\inv(w)}q^{\maj(w)}.
\]
As for the polynomials $s_{n,k}$ we need to introduce a further catalytic parameter.
If $w\in \mathscr R_{n,k}$ we let $r(w)$ be $n-w(x_n)$, where $x_n$ is the top element of $R_{n,k}$.
We let
\[
r_{n,k}(q,t)\eqdef \sum_{w\in \mathscr R_{n,k}}\eps^{\inv(w)}q^{\maj(w)}t^{r(w)}.
\]
These polynomials satisfy the following recursion.
\begin{lem}\label{recr}
For all $n\geq k+2$ we have
\[
r_{n,k}(q,t)=\frac{1}{1+t}\Big(t(1-q^{n-1})r_{n-1,k}(q,-t)+(1+t(- qt)^{n-1})r_{n-1,k}(q,1)\Big)
\]
\end{lem}
\begin{proof}
This proof is similar (and actually much simpler) to that of Theorem \ref{recur} and so we briefly sketch it.
There is a bijection $w\rightarrow (w_0,w(x_n))$ between $\mathscr R_{n,k}$ and $\mathscr R_{n-1,k}\times [n]$ satisfying $\inv(w)=\inv(w_0)+n-w(x_n)$, $r(w)=n-w(x_n)$, $\maj(w)=\maj(w_0)$ if $w_0(n-1)< w(n))(n-1)$, and $\maj(w)=\maj(w_0)+n-1$ otherwise. Therefore
\begin{align*}
r_{n,k}(q,t)&=\sum_{w_0\in \mathscr R_{n-1,k}}\eps^{\inv(w_0)}q^{\maj(w_0)}\Big(q^{n-1}\sum_{i=1}^{w_0(n-1)}(- t)^{n-i}+\sum_{i=w_0(n-1)+1}^n (- t)^{n-i}\Big)\\
&=\sum_{w_0\in \mathscr R_{n-1,k}}\eps^{\inv(w_0)}q^{\maj(w_0)}\Big(q^{n-1}\frac{(- t)^{n-w_0(n-1)}-(- t)^n}{1+ t}+\frac{1-(- t)^{n-w_0(n-1)}}{1+ t}\Big)\\
&=\frac{1}{1+t}\Big( t(1-q^{n-1})r_{n-1,k}(q,-t)+(1+t(-qt)^{n-1})r_{n-1,k}(q,1)\Big).\qedhere
\end{align*}

\end{proof}
Despite the polynomials $s_{n,k}$, the polynomials $r_{n,k}(q,t)$ admit a nice factorization only if $k$ is odd.
\begin{prop}\label{rnkodd}
Let $k$ be odd and $n>k$. Then
\[
r_{n,k}(q,t)=[k+1]_{-q}[k+2]_{q}\cdots [n-1]_{\eps^n q}[n]_{\eps^{n+1}qt}.
\]
\end{prop}
\begin{proof}
The case  $n=k+1$ is an easy verification. Since this initial condition plus the recursion of Lemma \ref{recr} uniquely determine $r_{n,k}(q,t)$ we only need to show that the claimed expression satisfies the recurrence. Substituting $r_{n,k}(q,t)$ as given by the Proposition and cancelling a common factor, it remains to check that
\[
   [n]_{\eps^{n+1}qt}=\frac{1}{t+1}\Big(t(1+\eps^{n-1}q)[n-1]_{\eps^{n-1}qt}+1+t(-qt)^{n-1}\Big),
\]
where we have also used that $(1-q^{n-1})=(1+\eps^{n-1}q)[n-1]_{\eps^{n}q}$. But the above right hand side may be rewritten as
\begin{align*}
\frac{1}{1+t}& \Big(t([n-1]_{\eps^{n-1}qt}+(-qt)^{n-1})+\eps^{n-1}qt[n-1]_{\eps^{n-1}qt}+1\Big)\\
&=\frac{1}{1+t}\Big(t[n]_{\eps^{n-1}qt}+[n]_{\eps^{n-1}qt}\Big)\\
&=[n]_{\eps^{n-1}qt}. \qedhere
\end{align*}
\end{proof}

If $k$ is even the polynomial $r_{n,k}(q,t)$ does not admit a nice factorization in general. Nevertheless, if $n$ is also even the specialization $r_{n,k}(q)$ can be easily computed thanks to the following result, whose proof is valid for all $k$.
\begin{prop}
   Let $n$ be even. Then
\[
   r_{n,k}(q)=[k+1]_{\eps^{k}q}[k+2]_{\eps^{k+1}q}\cdots [n-1]_q[n]_{-q}.
\]
\end{prop}
\begin{proof}
   We generalize here an idea appearing in \cite{Wa}. Consider the following bijection $\phi$ of $\mathscr R_{n,k}$. If $w\in \mathscr R_{n,k}$ is such that the vertices labelled by $2i-1$ and $2i$ are either adjacent or are not comparable for all $i\in [\frac{n}{2}]$ we let $\phi(w)=w$. Otherwise let $i$ be the minimum index such that $2i-1$ and $2i$ are comparable but not adjacent and we let $\phi(w)$ be the labelling obtained by exchanging the labels $2i$ and $2i-1$. It is clear that $\phi$ is an involution on $\mathscr R_{n,k}$ and that if $\phi(w)\neq w $ then $\eps^{\inv(w)}=-\eps^{\inv(\phi(w))}$ and $\maj(w)=\maj(\phi(w))$. Therefore, computing $r_{n,k}(q)$ we can restrict the sum on the fixed points of $\phi$.
If $w$ is fixed by $\phi$ then among the $\frac{n}{2}$ pairs of labels of the form $\{2i-1,2i\}$ there are $\lfloor \frac{n+1-k}{2}\rfloor$ pairs in adjacent positions, and $\lfloor\frac{k}{2}\rfloor$ pairs which appear in the teeth of the rake. Therefore we can construct an element $\bar w\in \mathscr R_{\frac{n}{2},\lfloor\frac{k}{2}\rfloor}$ in the following way: if $2i-1$ and $2i$ appear in the teeth of $w$ then $i$ is a label of a tooth of  $\bar w$. The other labels are inserted in such way that $w^{-1}(2i)<w^{-1}(2j)$ if and only if $\bar w^{-1}(i)<\bar w^{-1}(j)$.
We also define a 0-1 vector $a(w)=(a_1(w),\ldots,a_d(w))$, with $d=\lfloor \frac{n+1-k}{2}\rfloor$, in the following way. Let $\{2i_1-1,2i_1\}, \ldots,\{2i_d-1,2i_d\}$ be the pairs of labels of adjacent vertices ordered in such way that $w^{-1}(2i_1)< w^{-1}(2i_2)<\cdots<w^{-1}(2i_d)$, and we let $a_j(w)=1$ if and only if $w^{-1}(2i_j)<w^{-1}(2i_j-1)$. It is a straightforward verification that the map $w\rightarrow (\bar w,a(w))$ is a bijection between the fixed points of $\phi$ and  $\mathscr R_{\frac{n}{2},\lfloor\frac{k}{2}\rfloor}\times \{0,1\}^d$ such that
\begin{itemize}
\item $\inv(w)\equiv \sum {a_i(w)} \mod 2;$
\item $\maj(w)=2\maj(\bar w)+a_1+\sum_{i=2}^d (k+2i-2)a_i $ if $k$ is odd;
\item $\maj(w)=2 \maj(\bar w)+\sum_{i=1}^d (k+2i-1)a_i$ if $k$ is even;
\end{itemize}
Therefore, if $k$ is odd, applying \eqref{rakeunsigned}, we have
\begin{align*}
r_{n,k}(q)&=\sum_{\bar w \in \mathscr R_{\frac{n}{2},\lfloor\frac{k}{2}\rfloor}}q^{2\maj(\bar w)}(1-q)(1-q^{k+2})(1-q^{k+4})\cdots(1-q^{n-1})\\
&=\Big[\frac{k+1}{2}\Big]_{q^2}\Big[\frac{k+3}{2}\Big]_{q^2}\cdots \Big[\frac{n}{2}\Big]_{q^2}(1-q)(1-q^{k+2})\cdots(1-q^{n-1})\\
&=[k+1]_q[k+2]_{-q}\cdots[n]_{-q}
\end{align*}
Similarly, if $k$ is even, we have
\begin{align*}
r_{n,k}(q)&=\sum_{\bar w \in \mathscr R_{\frac{n}{2},\lfloor\frac{k}{2}\rfloor}}q^{2\maj(\bar w)}(1-q^{k+1})(1-q^{k+3})\cdots(1-q^{n-1})\\
&=\Big[\frac{k+2}{2}\Big]_{q^2}\Big[\frac{k+4}{2}\Big]_{q^2}\cdots \Big[\frac{n}{2}\Big]_{q^2}(1-q^{k+1})(1-q^{k+3})\cdots(1-q^{n-1})\\
&=[k+1]_{-q}[k+2]_{q}\cdots[n]_{-q}.
\end{align*} 
\end{proof}

\section{Mahonian distribution on parabolic quotients in complex reflection groups}
In this section we consider the infinite family of complex reflection groups  $G(r,p,n)$. We let $G=G(r,p,n)$ and  $G^*=G(r,n)/C_p$ and we recall from \cite{Ca1} that
$$
\sum_{g\in G^*}q^{\fmaj(g)}=[d_1]_q[d_2]_q\cdots[d_n]_q,
$$
where $d_i=ri$ if $i<n$ and $d_n=\frac{rn}{p}$ are the fundamental degrees of $G$. It is therefore natural to look at the distribution of the flag-major index on sets of cosets representatives for some special subgroups of $G^*$, in order to generalize Equation \eqref{greta} to these groups. With this in mind we need to extend the ideas appearing in \cite{St}, and in particular the use of $P$-partitions, in this context, using some of the tools developed in \cite{Ca1} and further exploited in \cite{BC}.
In particular, we recall the following result (see \cite[Theorem 8.3]{Ca1} and \cite[Lemma 5.1]{BC}).
\begin{lem}\label{bij}
   The map
\begin{align*}
 G^*\times \mathcal P_n\times \{0,1,\ldots,p-1\} & \longrightarrow  \mathbb N^n\\
(g,\lambda,h)&\mapsto f=(f_1,\ldots,f_n),
\end{align*}
where
$f_i=\lambda_{|g^{-1}(i)|}(g)+r\lambda_{|g^{-1}(i)|}+h\frac{r}{p}$ for all $i\in[n]$, is a bijection.
And in this case we say that $f$ is $g$-compatible.
\end{lem}

For $g\in G^*$ we let $S_g$ be the set of $g$-compatible vectors in $\mathbb N^n$.
\begin{lem}\label{F_g}We have
$$
F_g(x_1,\ldots,x_n)\eqdef\sum_{f\in S_g}x_1^{f_1}\cdots x_n^{f_n}=\frac{x_{|g_1|}^{\lambda_1(g)}\cdots x_{|g_n|}^{\lambda_n(g)}}{(1-x_{|g_1|}^r)(1-x_{|g_1|}^rx_{|g_2|}^r)\cdots(1-x_{|g_1|}^r\cdots x_{|g_{n-1}|}^r)\cdot (1-x_{|g_1|}^{r/p}\cdots x_{|g_n|}^{r/p} )}.$$
\end{lem}

\begin{proof} By Lemma \ref{bij} we have
   \begin{align*}
\sum_{f\in S_g}x_1^{f_1}\cdots x_n^{f_n}&= \sum_{\lambda\in \mathcal P_n} \sum_{h=0}^{p-1}x_{|g_1|}^{\lambda_1(g)+r\lambda_1+h\frac{r}{p}}\cdots  x_{|g_n|}^{\lambda_n(g)+r\lambda_n+h\frac{r}{p}} \\
&= x_{|g_1|}^{\lambda_1(g)}\cdots x_{|g_n|}^{\lambda_n(g)}\sum_{\lambda\in \mathcal P_n}x_{|g_1|}^{r\lambda_1}\cdots x_{|g_n|}^{r\lambda_n}\sum_{h=0}^{p-1}(x_1\cdots x_n)^{h\frac{r}{p}}\\
&=x_{|g_1|}^{\lambda_1(g)}\cdots x_{|g_n|}^{\lambda_n(g)}\frac{1}{(1-x_{|g_1|}^r)(1-x_{|g_1|}^rx_{|g_2|}^r)\cdots(1-x_{|g_1|}^r\cdots x_{|g_{n}|}^r)} \frac{1-x_1^r\cdots x_n^r}{1-x_1^{\frac{r}{p}}\cdots x_n^{\frac{r}{p}}},
   \end{align*}
and the result follows.
\end{proof}
Let $$\mathbb N^n_{(r,p)}\eqdef\{(f_1,\ldots,f_n)\in \mathbb N^n:\, f_1\equiv f_2\equiv \cdots \equiv f_n\equiv h\frac{r}{p}\mod r\textrm{ for some }h=0,1,\ldots,p-1\}$$ 
and $\mathcal A=\{(f_1,\ldots,f_n)\in \mathbb N^n:\, f_1\geq f_2\geq \cdots \geq f_{k} \textrm{ and }(f_1,\ldots,f_{k})\in \mathbb N^{k}_{(r,p)}\}$. We will show that the set $\mathcal A$ consists of all $g$-compatible vectors in $\mathbb N^n$ as $g$ varies in a suitable subset of $G^*$. Before proving this we need the following preliminary result. 

\begin{lem}\label{colori}
Let $g\in G^*$. Then $(\lambda_1(g)+ \lambda_{|g(1)|}(g^{-1}),\ldots,\lambda_n(g)+ \lambda_{|g(n)|}(g^{-1}))\in \mathbb N^n_{(r,p)}$.
\end{lem}
\begin{proof}
We recall that if $|g|=\si$ then $g=[\si_1^{\lambda_1(g)},\ldots,\si_n^{\lambda_n(g)}]$ and hence also $g^{-1}=[\tau_1^{\lambda_1(g^{-1})}, \ldots,\tau_n^{\lambda_n(g^{-1})}]$, where $\tau=\si^{-1}$. Therefore 
$$
g^{-1}g=[1^{\lambda_1(g)+\lambda_{\si_1}(g^{-1})}, \ldots,n^{\lambda_n(g)+\lambda_{\si_n}(g^{-1})}].
$$
Since this is the identity element in the group $G^*$, we deduce that there exists $h\in \{0,1,\ldots,p-1\}$ such that ${\lambda_i(g)+\lambda_{\si_i}(g^{-1})}\equiv h\frac{r}{p}$ for all $i\in [n]$, and the proof is complete.
\end{proof}
For $k<n$ we let $C_{k}\eqdef\{[\sigma_1^0,\sigma_2^0,\ldots,\sigma_{k}^0,g_{k+1},\ldots,g_n]\in G^*: \,\sigma_1<\cdots<\sigma_{k}\}$. We observe that the subgroup of $G^*$ given by $\{g\in G^*:\, g=[\si_1^{z_1},\ldots,\si_k^{z_k},(k+1)^0,(k+2)^0,\ldots,n^0\}$ is isomorphic to $G(r,k)$ for all $k<n$. Moreover, we may observe that $C_{k}$ contains exactly $p$ representatives for each right coset of $G(r,k)$ in $G^*$ (we prefer here to consider right instead of left cosets to be consistent with the notation and the results in \cite{G} and \cite{Pa} that we are going to generalize). In particular, if $p=1$, we have that $C_{k}$ is a complete system of representatives of the cosets of $G(r,k)$ in $G(r,n)$.
\begin{prop}\label{calA}
   Let $f\in \mathbb N^n$. Then $f\in \mathcal A$ if and only if $f$ is $g^{-1}$-compatible for some $g\in C_{k}$.
\end{prop}
\begin{proof}
 Recall that, if $|g|=\sigma$, then $g=[\sigma_1^{\lambda_1(g)},\sigma_2^{\lambda_2(g)},\ldots,\sigma_n^{\lambda_n(g)}]$. In particular,  $g\in C_{k}$ if and only if $\si_1<\si_2<\cdots<\si_{k}$ and  
\begin{equation}\label{ev}
 (\lambda_1(g), \lambda_2(g),\ldots,\lambda_{k}(g))\in \mathbb N^{k}_{(r,p)}.
\end{equation}
Let $f\in \mathbb N^n$ be $g^{-1}$-compatible for some $g\in G^*$. Then, by Lemma \ref{bij}, there exist $\lambda\in \mathcal P_n$ and $h\in \{0,\ldots,p-1\}$ such that
\begin{equation} \label{fg} f_i=\lambda_{\si_i}(g^{-1})+r\lambda_{\si_i}+h\frac{r}{p}, \qquad 1\leq i\leq k\\
\end{equation}
The proof will follows immediately from the following two claims and  Equation \eqref{ev}.
\begin{itemize}
\item[Claim 1:] $f_1\geq f_2\geq\cdots \geq f_k$ if and only if $\si_1<\si_2<\cdots<\si_k$. Since $\lambda(g^{-1})$ and $\lambda$ are both partitions, it is clear that Equations \eqref{fg} imply that   if  $\si_1<\si_2<\cdots<\si_{k}$ then $f_1\geq f_2 \geq \cdots \geq f_{k}$. For the same reason, if $f_{i}> f_{i+1}$ then $\si_i<\si_{i+1}$. Finally, assume that $f_i=f_{i+1}$. This immediately implies  that $\lambda_{\si_i}(g^{-1})=\lambda_{\si_{i+1}}(g^{-1})$ and a moment's thought based on the definition of the statistics $\lambda_i(g^{-1})$ will show that this also implies $\si_i<\si_{i+1}$.
\item [Claim 2:] $(f_1\geq f_2\geq\cdots \geq f_k)\in \mathbb N^{k}_{(r,p)} $ if and only if $(\lambda_1(g),\ldots,\lambda_{k}(g))\in \mathbb N^{k}_{(r,p)}$. By Equations \eqref{fg} we clearly have that $(f_1,\ldots,f_{k})\in \mathbb N^{k}_{(r,p)}$ if and only if $(\lambda_{|g_1|}(g^{-1}),\lambda_{|g_2|}(g^{-1}),\ldots,\lambda_{|g_{k}|}(g^{-1}))\in \mathbb N^{k}_{(r,p)}$. By Lemma \ref{colori} this is also equivalent to $(\lambda_1(g),\ldots,\lambda_{k}(g))\in \mathbb N^{k}_{(r,p)}$ and the proof is complete.
\end{itemize}
\end{proof}
We are now ready to state and prove the main result of this section.
\begin{thm}\label{grpn}Let $G=G(r,p,n)^*$. Then
$$
\sum_{g\in C_{k}}q^{\fmaj(g^{-1})}=[p]_{q^{kr/p}}[r(k+1)]_q\cdots [r(n-1)]_q[{rn/p}]_q.
$$ 
\end{thm}
\begin{proof} Consider the formal power series $$G(q)\eqdef\sum_{f\in \mathcal A}q^{|f|},$$
where, if $f=(f_1,\ldots,f_n)$, we let $|f|=f_1+\cdots+f_n$.
We compute the series $G(q)$ in two different ways.
First, by Lemma \ref{F_g} and Proposition \ref{calA} we have
\begin{align*}
G(q)&= \sum_{g\in C_{k}}F_{g^{-1}}(q,\ldots,q) \\
&= \sum_{g\in C_{k}} \frac{q^{\lambda_1(g^{-1})}\cdots q^{\lambda_n(g^{-1})}}{(1-q^r)(1-q^{2r})\cdots (1-q^{r(n-1)} )(1-q^{rn/p})}\\
&= \frac{\sum_{g\in C_{k}}q^{\fmaj(g^{-1})}}{(1-q^r)(1-q^{2r})\cdots  (1-q^{r(n-1)} )(1-q^{rn/p})}.
\end{align*}
Now we compute directly $G(q)$ using the definition of $\mathcal A$. We have
\[
 G(q)=\frac{1}{(1-q^r)(1-q^{2r})\cdots (1-q^{kr})}\frac{(1-q^{kr})}{(1-q^{kr/p})}\frac{1}{(1-q)^{n-k}}
\]
and therefore
$$
\sum_{g\in C_{k}}q^{\fmaj(g^{-1})}=\frac{1-q^{kr}}{1-q^{kr/p}}\frac{(1-q^{(k+1)r})\cdots (1-q^{(n-1)r})(1-q^{nr/p})}{(1-q)^{n-k}}.\qedhere
$$
\end{proof}
\begin{cor}\label{grn}
 If $G=G(r,n)$, then $C_{k}$ is a system of coset representatives for the subgroup $G(r,k)$ and
$$
\sum_{g\in C_{k}}q^{\fmaj(g^{-1})}=[r(k+1)]_q[r(k+2)]_q\cdots [rn]_q.
$$
\end{cor}

Equation \eqref{greta} has been rediscovered by Panova in \cite{Pa} to present a simple combinatorial proof of a related result that was discovered by Garsia in \cite{G}, using nice relations of these objects with symmetric functions. Corollary \ref{grn} can be easily used to find a natural $G(r,n)$ counterpart of these results. We first need some further notation. If $g\in G(r,n)$ we let
$$
is_z(g)\eqdef \max \{j:\, \exists 1\leq i_1<\cdots<i_j\leq n \textrm{ with } z_{i_1}(g)=\cdots z_{i_j}(g)=z\textrm{ and } |g(i_1)|<\cdots<|g(i_j)|\},
$$
be the maximum length of a homogeneous increasing subsequence of $g$ of color $z$.
Then we let 
$$\Pi_{r,n,k}\eqdef \{g=[\si_1^{0},\ldots,\si_{n-k}^{0},\si_{n-k+1}^{z_{n-k+1}},\ldots,\si_n^{z_n}]\in G(r,n):\si_1<\cdots<\si_{n-k}\textrm{ and }is_0(g)=n-k\}. $$
The following result generalizes  \cite[Theorems 1 and 2]{Pa} and the main results in \cite{G}.
\begin{thm}
   If $n\geq 2k$ we have that
$$
\sum_{g\in \Pi_{r,n,k}}q^{\fmaj(g^{-1})}=\sum_{i=0}^k (-1)^i\binom{k}{i}[r(n-i+1)]_q[r(n-i+2)]_q\cdots[rn]_q.
$$
\end{thm}
   The proof of this result for $r=1$ in \cite{Pa} makes use of Equation \eqref{greta} and on a cute use of the principle of inclusion-exclusion on some sets of standard tableaux based on the Robinson-Schensted correspondence. This proof can be generalised to the general case of groups $G(r,n)$ using Corollary \ref{grn} and the generalised version of the Robinson-Schensted correspondence \cite{SW}. This is why we do not present this proof and we refer the reader to \cite[\S 10]{Ca1} for further details on the generalized Robinson-Schensted correspondence. 

We conclude with some open problems arising from this work.
\begin{prob}
   Let $J'=[k]$ and $\mathscr S_{n,k}'\eqdef \textrm{\raisebox{6pt}{$_{J'}$}}\!S_n=\{\si\in S_n:\si^{-1}(1)<\cdots <\si^{-1}(k)\}$. 
Numerical evidence suggests that
$$
\sum_{\si \in \mathscr S_{n,k}'}(-1)^{\inv(\si)}q^{\maj(\si)}=\sum_{u\in \mathscr S_{n,k}}(-1)^{\inv(\si)}q^{\maj(\si)}
$$ 
if $n$ is even or $k$ is odd. These are the same cases for which $s_{n,k}(q)=r_{n,k}(q)$. Give a (possibly bijective) proof of these equidistributions.
\end{prob}
\begin{prob}
   Unify the main results of this work in a unique statement, i.e. compute the polynomials
$$
\sum_{g\in C_{k}}\eps^{\inv(|g|)}q^{\fmaj(g^{-1})},
$$
where $C_{k}$ is a system of coset representatives of $G(r,k)$ in $G(r,n)$ including as particular cases Theorems \ref{main} and Corollary \ref{grn}.
\end{prob}

\noindent \emph{Acknowledgement.} I would like to thank an anonymous referee for suggesting, in his report on a preliminary version of this paper, to look for a possible relation between signed mahonian distributions on parabolic quotients and forests.

\end{document}